\documentclass[reqno]{amsart} 
\usepackage{amsfonts, amsmath, amsthm, amssymb, latexsym, graphicx, geometry, xcolor, colortbl,  mathtools}

\theoremstyle{plain}
\newtheorem{theorem}{Theorem}[section]

\newtheorem{lemma}[theorem]{Lemma}

\theoremstyle{definition}

\numberwithin{equation}{section}

\newcommand{\lkn}{\left(1-\frac kN\right)}
\newcommand{\ab}{\alpha+\beta}

\title[Jacobi ensembles]{On the extremal eigenvalues of Jacobi ensembles at zero temperature}

\author{Kilian Hermann, Michael Voit} 
\address{Fakult\"at Mathematik, Technische Universit\"at Dortmund,
          Vogelpothsweg 87,
          D-44221 Dortmund, Germany}
\email{kilian.hermann@math.tu-dortmund.de, michael.voit@math.tu-dortmund.de}


\begin{document}
\date{\today}

\begin{abstract}
  For the $\beta$-Hermite, Laguerre, and Jacobi ensembles of  dimension $N$ there
  exist central limit theorems for the freezing case $\beta\to\infty$ such that the
  associated means and covariances can be expressed in terms of the associated Hermite, Laguerre,
  and Jacobi polynomials of order $N$ respectively as well as via the associated dual polynomials
  in the sense of de Boor and Saff.
  In this paper we derive limits for $N\to\infty$ for the covariances of the $r\in\mathbb N$ largest (and smallest) eigenvalues for these 
  frozen Jacobi ensembles in terms of Bessel functions. These results correspond
  to the hard edge analysis in the frozen Laguerre cases by Andraus and  Lerner-Brecher and to known results for finite $\beta$.
  \end{abstract}

\maketitle

\section{Introduction}

The $\beta$-Hermite, Laguerre, and Jacobi ensembles of  dimension $N$ depend on 1,2, and 3 parameters respectively.
The associated distributions on the 
corresponding  Weyl chambers and alcoves in $\mathbb R^N$ can be described in suitable coordinates as follows:

In the $\beta$-Hermite case, for the parameter $\beta>0$, by
\begin{equation}\label{densitiybetaHermite}
d\mu_H^\beta(x)=	C_{\beta,N,H}\prod_{1\leq i<j\leq N}(x_j-x_i)^\beta\prod_{i=1}^N e^{-x_i^2/2}\> dx
\end{equation}
on the Weyl chamber $C_N^A:=\{x\in\mathbb R^N:\> x_1\le x_2\le \cdots\le x_N\}$ of type $A$.

In the $\beta$-Laguerre case, for the parameters $\beta,\nu>0$, by
\begin{equation}\label{densitiybetaLaguerre}
d\mu_L^{\beta,\nu}(x)=	C_{\beta,\nu, N,L}\prod_{1\leq i<j\leq N}(x_j^2-x_i^2)^{2\beta}\prod_{i=1}^N\left(  x_i^{2\nu\beta} e^{-x_i^2/2}\right)\> dx
\end{equation}
on the Weyl chamber $C_N^B:=\{x\in\mathbb R^N:\> 0\le x_1\le x_2\le \cdots\le x_N\}$ of type $B$.

In the $\beta$-Jacobi case, for the parameters $\kappa>0$ and $\alpha,\beta >-1$, by
\begin{equation}\label{densitiybetaJacobi}
  d\mu_J^{\kappa,\alpha,\beta}(x)=	C_{\kappa,\alpha,\beta, N,J}
\prod_{1\leq i<j\leq N}(x_j-x_i)^{\kappa}\prod_{i=1}^N(1-x_i)^{\frac{(\alpha+1) \kappa}{2}-\frac12}(1+x_i)^{\frac{(\beta+1)\kappa}{2}-\frac12}\> dx
\end{equation}
on the alcove $A_N:=\{x\in\mathbb R^N: \> -1\le x_1\le\cdots\le x_N\le 1\}$.

Here, 	$C_{\beta,N,H}, 	C_{\beta,\nu, N,L}, 	C_{\kappa,\alpha,\beta, N,J}>0$ are known Selberg-type normalization constants; see e.g.
\cite{M, FW}.

Our parametrization in the Jacobi case
(\ref{densitiybetaJacobi}) is unusual in random matrix theory (see e.g.~\cite{BG, F1, F2, Ji, K, KN, Li, N, Wi}), where our $\kappa$
is usually denoted by $\beta$. Our parametrization is motivated by notations in special functions, as in
the notation (\ref{densitiybetaJacobi})
the Jacobi ensembles are closely related with the
roots of the classical one-dimensional Jacobi polynomials $P_N^{(\alpha,\beta)}$. 

For all three classes of ensembles there exist freezing central limit theorems for $\beta\to\infty$ in the first two cases
and $\kappa\to\infty$ in the Jacobi case, where
 the further parameters are fixed, and where the additive corrections as well as the entries of the $N\times N$-dimensional inverse matrices of
the
covariance matrices can be computed explicitely in terms of the roots of the associated classical Hermite polynomial $H_N$,  Laguerre
polynomial $L_N^{(\nu-1)}$, and  Jacobi polynomial $P_N^{(\alpha,\beta)}$ respectively; see \cite{DE, V, AV, AHV, GK} for the Hermite and Laguerre case, and
\cite{HV} for the Jacobi case. Moreover, the eigenvalues and eigenvectors of these inverse covariance matrices are known (see
\cite{AV, HV}). This and the theory of dual orthogonal polynomial in the sense of de Boor and Saff (see \cite{BS, I, VZ})
was used in \cite{AHV} to compute  the covariance matrices in all three cases. Different direct appraches to the covariance matrices
in the Hermite and Laguerre case are given in \cite{DE, GK, L}.

In all these cases, it is possible to describe the limits for the extremal entries of these  covariance matrices for $N\to\infty$.
In the Hermite case as well as for the soft edge Laguerre case (i.e., the $r$ largest eigenvalues), the corresponding
roots of the corresponding Hermite and Laguerre polynomials can be described (after renorming) via the corresponding roots 
of the Airy function, and the entries of the covariance matrices tend to some integrals involving the Airy function; see \cite{GK, AHV}
for details. On the other hand, for the  hard edge Laguerre case (i.e., the $r$-th smallest eigenvalues), the Airy function must be
replaced by some  Bessel function $J_\alpha$ ; see \cite{A, L}.

In this paper we mainly derive corresponding results for the $r$-th largest eigenvalues of Jacobi ensembles where again 
the  Bessel function $J_\alpha$ and its roots appears.
Corresponding results  for the $r$-th smallest eigenvalues can be obtained  by using the transformation $x\mapsto -x$
where $\alpha,\beta$ are exchanged.

In order to state the main results in the Jacobi case, let 
$-1<z_{1,N}^{(\alpha,\beta)}<...<z_{N,N}^{(\alpha,\beta)}<1$ be the  ordered zeros of the Jacobi polynomial $P_N^{(\alpha,\beta)}$ (as defined e.g.~in \cite{S}), and
 let  $z_N^{(\alpha,\beta)}=(z_{1,N}^{(\alpha,\beta)},...,z_{N,N}^{(\alpha,\beta)})\in A_N$ be the vector with these zeros as entries.

 Furthermore, according to \cite{W}, we consider the  Bessel functions
 $$		J_\alpha(z)=\sum_{m=0}^{\infty}\frac{(-1)^m(\frac{z}{2})^{\alpha+2m}}{m!\Gamma(\alpha+m+1)} $$
 on $]0,\infty]$ and
     the increasing sequence  $(j_{\alpha,r})_{r\ge1}$ of positive zeroes of  $J_\alpha$. We then have:

  \begin{theorem}\label{asymptotics-jacobi-intro}
    Let $\alpha,\beta>-1$. For
  $\kappa>0$ let $X_\kappa$ be $N$-dimensional Jacobi-type
  random variables with the densities (\ref{densitiybetaJacobi}) on the alcoves $A_N$. Then
  $\sqrt{\kappa}\left({X_\kappa}-z^{(\alpha,\beta)}_N\right)$ converges for $\kappa\to\infty$
  in distribution to some centered $N$-dimensional normal distribution
  $N(0,\Sigma_N)$. For $r,s\in\mathbb{N}$, the entries of these covariance matrices 
  $\Sigma_{N}=:\left(\sigma_{i,j}\right)_{i,j=1,...,N}$
satisfy
\begin{equation}\label{limit-jacobi-algebraic-intro}
	\lim\limits_{N\rightarrow\infty}  \sigma_{N-r+1,N-s+1}=
	\frac{1}{4 j_{\alpha,r} j_{\alpha,s}
          J_\alpha^\prime(j_{\alpha,r}) J_\alpha^\prime(j_{\alpha,s})}
        \int_{0}^{1}\frac{y}{1-y^2}J_\alpha(j_{\alpha,r}y)J_\alpha(j_{\alpha,s}y) \> dy.
\end{equation}
  \end{theorem}

  We shall prove Theorem \ref{asymptotics-jacobi-intro} in the next section. We point out that
  corresponding results  for the opposite components of Jacobi ensembles can be obtained immediately 
  by the transformation $x\mapsto -x$  where $\alpha,\beta$ are exchanged.
  We do not state this result separately. Moreover, Theorem \ref{asymptotics-jacobi-intro}
  can be also stated for Jacobi ensembles in trigonometric coordinates, which fit to the theory of Heckman-Opdam Jacobi polynomials
  \cite{HO, HS}; see Section 2 for more details.

  Our approach to Theorem \ref{asymptotics-jacobi-intro} for Jacobi ensembles can be also transferred to
  the hard edge analysis (i.e., for the smallest components) of Laguerre ensembles where the following holds:
  
	\begin{theorem}\label{limit-cov-laguerre-intro}
	Let $\nu>0$.  For $\beta>0$ consider Laguerre-type random variables $X_\beta$ on $C_N^B$ 
  with the densities (\ref{densitiybetaLaguerre}). Let 
 $0< z_{1,N}^{(\nu-1)}<\ldots<  z_{N,N}^{(\nu-1)}$ be the ordered zeroes of
 the Laguerre polynomial $L_N^{(\nu-1)}$ and
$${\bf r}:={\bf r}(\nu):=(r_1,\ldots, r_N):=\left(\sqrt{ z_{1,N}^{(\nu-1)}}, \ldots, \sqrt{ z_{N,N}^{(\nu-1)}}\right).$$
  Then  $ X_\beta-\sqrt{2\beta} \cdot {\bf r}$
 	converges  for $\beta\to\infty$ in distribution to  some normal distribution 
        $N(0,\Sigma_N)$. For $r,s\in\mathbb{N}$, the entries of the
        covariance matrices $\Sigma_{N}=\left(\sigma_{i,j}^{(N)}\right)_{i,j=1,...,N}$
satisfy
\begin{align*}
		\lim_{N\to\infty} N\sigma_{r,s}^{(N)}=
		\int_{0}^{1}\frac{ J_{\nu-1}(j_{\nu-1,r}\sqrt{1-y}) J_{\nu-1}(j_{\nu-1,s}\sqrt{1-y})}{2y\cdot J_{\nu-1}^\prime(j_{\nu-1,r})J_{\nu-1}^\prime(j_{\nu-1,s})} \>    dy.
		\end{align*}
\end{theorem}

  This result was already studied in \cite{A, L}: Lerner-Brecher \cite{L}
  uses a different, more general and complicated approach via  Laguerre-type interlacing results  while the approach in   Andraus \cite{A}
  is very similar to that in \cite{AHV} and in this paper. However, Andraus \cite{A} claims at some stage
  of his proof some uniform convergence result
  which we are not able to verify. He needs this result in order to interchange some limit with an integration. We  explain briefly in Section 3
  that  this exchange can be made precise with our methods of Section 2. We also mention that in \cite{CDM} the same method as in \cite{L}
  is applied to the hard edge analysis of discrete particle systems associated with deformed Young diagrams and Jack polynomials.

  We finally mention that Theorem \ref{asymptotics-jacobi-intro}  can be also stated for the extremal cases
  $\alpha=-1$ and/or $\beta=-1$ where then normal distributions on half and quarter spaces appear as limits;
  see \cite{He}. It is likely that Theorem \ref{asymptotics-jacobi-intro} can be extended also to such degenerated cases.

  Parts of this paper are contained in the thesis \cite{He}.
  
  \section{Proof for the extremal covariances of Jacobi ensembles}

In this section we prove Theorem \ref{asymptotics-jacobi-intro} as well its variant for Jacobi ensemles in trigonometric coordinates.
  
As before, let 
$-1<z_{1,N}^{(\alpha,\beta)}<...<z_{N,N}^{(\alpha,\beta)}<1$ be the  ordered zeros of the Jacobi polynomial $P_N^{(\alpha,\beta)}$
(as defined e.g.~in \cite{S}), and
let  $z_N^{(\alpha,\beta)}=(z_{1,N}^{(\alpha,\beta)},...,z_{N,N}^{(\alpha,\beta)})\in A_N$ be the vector with these zeros as entries.

The following central limit theorem is essential in this section; see Theorem 2.6 of \cite{HV}:

\begin{theorem}\label{theorem:cltjacobi}
  For $\alpha,\beta >-1$ and $\kappa>0$ let $X_\kappa$ be $N$-dimensional Jacobi-type
  random variables with the densities (\ref{densitiybetaJacobi}). Then
  $\sqrt{\kappa}\left({X_\kappa}-z^{(\alpha,\beta)}_N\right)$ converges for $\kappa\to\infty$
  in distribution to the centered $N$-dimensional normal distribution
  $N(0,\Sigma_N)$ where the inverse  
	$\Sigma_N^{-1}=:S_N=(s_{i,j})_{i,j=1,...,N}$ of the
covariance matrix $\Sigma_N$ is given by
	\begin{align*}
	s_{i,j}=
	\begin{cases}\sum_{l=1,\ldots,N; l\ne j}
	  \frac{1}{(z_{j,N}^{(\alpha,\beta)}-z_{l,N}^{(\alpha,\beta)})^2}+\frac{\alpha+1}{2}\frac{1}{(1-z_{j,N}^{(\alpha,\beta)})^2}+\frac{\beta+1}{2}\frac{1}{(1+z_{j,N}^{(\alpha,\beta)})^2}
          &\textit{ for }i=j\\
	\frac{-1}{(z_{i,N}^{(\alpha,\beta)}-z_{j,N}^{(\alpha,\beta)})^2}&\textit{ for }i\neq j
	\end{cases}
	\end{align*}
\end{theorem}

As in \cite{HV, AHV} we now rewrite the densities
 (\ref{densitiybetaJacobi}) and thus Theorem \ref{theorem:cltjacobi}  in  trigonometric coordinates, i.e.,
we define the probability measures $\tilde\mu_J^{\kappa,\alpha,\beta}(x)$ on the trigonometric alcoves
$$\tilde A_N:=\{t\in\mathbb R^N|\frac{\pi}{2}\geq t_1\geq...\geq t_N\geq 0 \}$$
with the Lebesgue densities
\begin{equation}\label{density-joint-trig}
\tilde 	C_{\kappa,\alpha,\beta, N,J}\cdot \prod_{1\leq i< j \leq N}\left(\cos(2t_j)-\cos(2t_i)\right)^{\kappa}
\prod_{i=1}^N\Bigl(\sin(t_i)^{\kappa(\alpha-\beta)}\sin(2t_i)^{\kappa(\alpha+1)}\Bigr)
\end{equation}
with some explicitly known normalization  $\tilde 	C_{\kappa,\alpha,\beta, N,J}>0$.
Then the $\mu_J^{\kappa,\alpha,\beta}(x)$ from (\ref{densitiybetaJacobi})  are the pushforwards of the 
  $\tilde\mu_J^{\kappa,\alpha,\beta}(x)$   under the transformations
\begin{align}\label{transformationtrigalgebraic}	
T_N: \tilde A_N\longrightarrow A_N, \quad T_N(t_1,\ldots,t_N):=(\cos(2t_1),\ldots, \cos(2t_N));
\end{align}
see \cite{HV}. This transformation is motivated by the theory of Heckman-Opdam on multivariate hypergeometric functions (see \cite{HO, HS}).

The CLT \ref{theorem:cltjacobi} then has the following form in which the eigenvectors and eigenvalues of the inverse
covariance matrix can be described more easily; see  \cite{HV, AHV}. We shall suppress the dependence of the zeroes $z_{j,N}^{(\alpha,\beta)}$ and the vector
$z_N^{(\alpha,\beta)}$ from the fixed parameters  $\alpha,\beta >-1$ from now on.

\begin{theorem}\label{theoremclttrig}
  For $\alpha,\beta >-1$ and $\kappa>0$ let $\tilde X_\kappa$ be $N$-dimensional trigonometric Jacobi-type 
  random variables with the densities (\ref{density-joint-trig}).
  Then  $$\sqrt{\kappa}(\tilde X_\kappa-T_N^{-1}(z_N))$$
 	converges  for $\kappa\to\infty$ in distribution to 
        $N(0,\tilde\Sigma_N)$ where the inverse
        $\tilde{\Sigma_N}^{-1}=:\tilde{S_N}=(\tilde{s}_{i,j})_{i,j=1,...,N}$ of the covariance matrix $\tilde{\Sigma_N}$  satisfies
 	\begin{align}
 	\tilde s_{i,j}=
 	\begin{cases}4\sum_{ l: l\ne j}
 	  \frac{1-z_{j,N}^2}{(z_{j,N}-z_{l,N})^2}+2(\alpha+1)\frac{1+z_{j,N}}{1-z_{j,N}}
          +2(\beta+1)\frac{1-z_{j,N}}{1+z_{j,N}} &\textit{ for }i=j\\
 	\frac{-4\sqrt{(1-z_{j,N}^2)(1-z_{i,N}^2)}}{(z_{i,N}-z_{j,N})^2}&\textit{ for }i\neq j.
 	\end{cases}
 	\end{align}
        The inverse covariance matrices $\tilde S_N$ here  and $S_N$ from Theorem \ref{theorem:cltjacobi} are related by
        \begin{equation}\label{connection-covariances}
          \tilde S_N = D S_N D \quad\text{with the diagonal matrix}\quad D=diag( 2\sqrt{1-z_{1,N}^2},\ldots,  2\sqrt{1-z_{N,N}^2}).
          	\end{equation}
        Furthermore, $\tilde S_N$ has the eigenvalues
 	$$\lambda_k=2k(2N+\alpha+\beta+1-k)>0 \quad\quad (k=1,\ldots,N).$$
 	The $\lambda_k$ have  eigenvectors of   the form
 	$$v_k:=\left(q_{k-1}(z_{1,N})\sqrt{1-z_{1,N}^2},\ldots,q_{k-1}(z_{N,N})\sqrt{1-z_{N,N}^2}\right)^T$$
 	with the finite orthonormal
         polynomials $(q_{k})_{k=0,\ldots,N-1}$   associated  with the discrete measure
 	\begin{equation}
 	\mu_{N,\alpha,\beta}:=(1-z_{1,N}^2)\delta_{z_{1,N}}+\ldots+(1-z_{N,N}^2)\delta_{z_{N,N}}.
 	\end{equation}
\end{theorem}

As explained in \cite{AHV}, the $(q_{k})_{k=0,\ldots,N-1}$  are the dual orthogonal polynomials of the Jacobi polynomials $P_N^{(\alpha,\beta)}$
  in the sense of de Boor and Saff \cite{BS}, and the covariance matrices $\tilde\Sigma_N=(\tilde \sigma_{i,j})_{i,j=1,\ldots ,N}$
  and thus $\Sigma_N=( \sigma_{i,j})_{i,j=1,\ldots ,N}$ above can be computed explicitly
  by matrix analysis.

 The aim of this paper is to derive a limit result for the variances $\sigma_{N+1-r,N+1-r}$ of the 
 $r$-th largest component of the algebraic frozen Jacobi ensembles (or, equivalent,  for the variances $\tilde\sigma_{N+1-r,N+1-r}$ of the 
 $r$-th smallest component of the trigonometric frozen Jacobi ensembles) for fixed $r\in \mathbb N$ and $N\to\infty$.
 These limits will be described in terms of $J_\alpha$  where these limits  are independent from $\beta$.
 We briefly recall e.g. from \cite{W} that $J_\alpha$  solves the differential equation
  \begin{align}\label{eq:BesselJHE}
		z^2\frac{d^2}{dz^2}y(z)+z\frac{d}{dz}y(z)+(z^2-\alpha^2)y(z)=0.
	\end{align}
We also recapitulate that the positive zeroes $j_{\alpha,r}$ of  $J_\alpha$ and $\theta_{r,N}$ of $P_N^{(\alpha,\beta)}(\cos\theta)$
  (with $\cos (2\theta_{r,N})=z_{N-r+1,N}$) 
  with $r\in\mathbb N$  are closely related. In particular, it is shown in \cite{G2} that for fixed $\alpha,\beta,r$ and for $N\to\infty$,
  \begin{equation}\label{asympt-zeros}
    \theta_{r,N}= \frac{j_{\alpha,r}}{\nu}+O(N^{-5}) \quad\text{with}\quad 
      \nu:=\sqrt{\left(N+\frac{\alpha+\beta+1}{2}\right)^2+\frac{1-\alpha^2-3\beta^2}{12}}=N+O(1).
\end{equation}
  This relation roughly explains the appearance of  $J_\alpha$ and its zeros in the limit in 
  Theorem \ref{asymptotics-jacobi-intro} which is part of the following result which is concerned with
  algebraic as well as trigonometric coordinates:

  \begin{theorem}\label{asymptotics-jacobi}
    For $\alpha,\beta>-1$ consider the covariance matrices $\tilde\Sigma_{N}=:\left(\tilde\sigma_{i,j}\right)_{i,j=1,...,N}$ and
 $\Sigma_{N}=:\left(\sigma_{i,j}\right)_{i,j=1,...,N}$
    of the trigonometric and algebraic
	$\beta$-Jacobi ensembles in the freezing regime. Then for $r,s\in\mathbb{N}$
	\begin{equation}\label{limit-jacobi-trigonometric}
	\lim\limits_{N\rightarrow\infty}{N^2}{\tilde \sigma_{N-r+1,N-s+1}}=
	\frac{1}{J_\alpha^\prime(j_{\alpha,r})J_\alpha^\prime(j_{\alpha,s})}
        \int_{0}^{1}\frac{y}{1-y^2}J_\alpha(j_{\alpha,r}y)J_\alpha(j_{\alpha,s}y) \> dy
	\end{equation}
and
    \begin{equation}\label{limit-jacobi-algebraic}
	\lim\limits_{N\rightarrow\infty}  \sigma_{N-r+1,N-s+1}=
	\frac{1}{4 j_{\alpha,r} j_{\alpha,s}
          J_\alpha^\prime(j_{\alpha,r}) J_\alpha^\prime(j_{\alpha,s})}
        \int_{0}^{1}\frac{y}{1-y^2}J_\alpha(j_{\alpha,r}y)J_\alpha(j_{\alpha,s}y) \> dy.
	\end{equation}    
  \end{theorem}

  As a preparation for the proof, we need some facts about the dual orthogonal polynomials 
  of the Jacobi polynomials $(P_N^{(\alpha,\beta)})_{N\ge0}$ in the sense of De Boor and Saff \cite{BS, I, VZ}.
  For this we consider the  orthonormal  Jacobi polynomials
 $(\tilde P_N:=\tilde P_N^{(\alpha,\beta)})_{N\ge0}$
   w.r.t. to the
probability measure with Lebesgue density
$$\frac{\Gamma(\alpha+\beta+2)}{2^{\alpha+\beta+1}\Gamma(\alpha+1)\Gamma(\beta+1)}  (1-x)^\alpha(1+x)^\beta$$
on $[-1,1]$.
The three-term recurrence  and orthogonality relation of $(P_N^{(\alpha,\beta)})_{N\ge0}$ from 
Sections 4.5 and 4.3 of \cite{S} then yields that the $\tilde P_N^{(\alpha,\beta)}$
 satisfy the three-term recurrence
\begin{equation}\label{normal-generic-3-term}
\tilde{P}_0=1,\> \tilde{P}_1(x)=a_1^{-1}(x-b_0),\>
x\tilde{P}_{N}(x)=a_{N+1}\tilde{P}_{N+1}(x)+b_{N}\tilde{P}_{N}(x)+a_{N}\tilde{P}_{N-1}(x)
\quad(N\ge1)
\end{equation}
with
\begin{align}\label{recursion-jabobi-normalized}
	a_N:=a_{N,\alpha,\beta}:&=\frac{\sqrt{4N(N+\alpha+\beta)(N+\alpha)(N+\beta)}}{\sqrt{(2N+\alpha+\beta+1)(2N+\alpha+\beta-1)}(2N+\alpha+\beta)},\\
	b_N:=b_{N,\alpha,\beta}:&=-\frac{\alpha^2-\beta^2}{(2N+\alpha+\beta)(2N+\alpha+\beta+2)}	
	\notag	\end{align}
for $N\ge 1$ and some $b_0\in\mathbb R$.
We now fix $N$ and follow \cite{VZ, AHV}  (see in particular Lemmas 4.2 and 4.3 of \cite{AHV}) and define the dual 
polynomials $(\tilde{Q}_{k,N})_{k=0,..,N-1}$ by the  three-term recurrence
 \begin{align}\label{dual-recurrence-normal}
&\tilde{Q}_{0,N}=1,\> \tilde{Q}_{1,N}(x)=a_{N-1}^{-1}(x-b_{N-1}),\>\\
&        x\tilde{Q}_{k,N}(x)=a_{N-k-1}\tilde{Q}_{k+1,N}(x)+b_{N-k-1}\tilde{Q}_{k,N}(x)+a_{N-k}\tilde{Q}_{k-1,N}(x)\quad
(1\leq k\leq N-2).
        \notag\end{align}
        The polynomials $(\tilde{Q}_{k,N})_{k=0,..,N-1}$ then are
orthonormal w.r.t.~the discrete measure
$$\sum_{i=1}^Nw_i^*\delta_{z_{i,N}}$$
with the dual Christoffel numbers
\begin{align}\label{dual-Christoffel}
w_i^*=\frac{\tilde P_{N-1}(z_{i,N})}{b_N\tilde P_{N}^\prime(z_{i,N})}>0\quad
(i=1,\ldots,N)
\end{align}
which satisfy $\sum_{i=1}^Nw_i^*=1$. By Lemma 4.5 in  \cite{AHV}, the normalized eigenvectors of the inverse covariance matrices
$\tilde S_N$ in Theorem \ref{theoremclttrig} can be now described as follows:

\begin{lemma}\label{jacobi-eigenvectors}
The  orthonormal  eigenvectors of 
$\tilde S_N$ associated with the eigenvalues $\lambda_j$ are given  by the vectors
        \begin{align}
                \frac{1}{\sqrt{h_N}}(\sqrt{1-z_{1,N}^2}\tilde Q_{j-1,N}(z_{1,N}),\ldots,\sqrt{1-z_{N,N}^2}\tilde{Q}_{j-1,N}(z_{N,N}))^T,\
1\leq j\leq N,
        \end{align}
        with
    	$$h_N:=\frac{{4N(N+\alpha)(N+\beta)(N+\alpha+\beta)}}{{(2N+\alpha+\beta)}^2{(2N+\alpha+\beta-1)}}   .$$ 
\end{lemma}

In summary, if we define the slightly renormalized orthogonal polynomials
$$Q_{j,N}:=\frac{1}{\sqrt{h_N}}\tilde Q_{j,N} \quad\quad (j=0,..,N-1),$$
we obtain 
$$T_N^T\tilde\Sigma_N T_N=\operatorname{diag}\left(\frac{1}{\lambda_1},\frac{1}{\lambda_2},...,\frac{1}{\lambda_N}\right)$$
with
$$\lambda_k=2k(2N+\alpha+\beta+1-k)>0 \quad(k=1,\ldots,N)$$
and the orthogonal matrix
\begin{align}\label{eq:TNorthogonalJHE}	
	T_N:= \left(\sqrt{1-z_{i,N}^2}Q_{j-1,N}(z_{i,N})\right)_{i,j=1,\ldots,N}
	\end{align}
where the sequence $({Q}_{k,N})_{k=0,..,N-1}$ satisfies
	\begin{align}\label{eq:3termjacobidualnorm2}
	xQ_{k,N}(x)&=a_{N-k}Q_{k-1,N}(x)+b_{N-k-1}Q_{k,N} (x)
	+a_{N-k-1 }Q_{k+1,N}x) \quad(k=0,\ldots,N-2)\\
        Q_{-1,N}&=0, \quad Q_{0,N}={\frac{1}{\sqrt{h_{N,\alpha,\beta}}}}.
	\notag\end{align}
The following result is the main ingredient of the proof of  Theorem \ref{asymptotics-jacobi}.

\begin{theorem} \label{theorem:localconvJHE}
	Let $\alpha,\beta>-1$, and $r\in\mathbb{N}$. For  $N\in\mathbb{N}$ with $N>r$ define the step functions
	\begin{align}\label{def:fNJHE}
	f_N(y):=f_{N,r}(y):=\sqrt{N}\sum_{k=0}^{N-1}\sqrt{1-z_{N-r+1,N}^2}Q_{k,N}(z_{N-r+1,N}){\bf 1}_{[\frac{k}{N},\frac{k+1}{N}[}(y).
	\end{align}
        as well as the function
\begin{equation}\label{limit-besseltype}
                  f^{(\alpha, r)}(y):= \frac{-\sqrt 2}{J_\alpha^\prime(j_{\alpha,r})} \sqrt{1-y}\cdot J_\alpha(j_{\alpha,r}(1-y)).
                  \end{equation}
        for $y\in[0,1[$. Then, for $N\to\infty$,
$|f_N-f|=O(N^{-1})$
locally uniformly on $[0,1[$.\end{theorem}

\begin{proof}
  The main idea of the proof is to rewrite the three-term recursion \eqref{eq:3termjacobidualnorm2} as a perturbated integral equation and to
apply Gronwall's lemma. In a second step we then solve the integral equation.

In the following considerations, let $y\in]0,1[$ and $k$ an integer with $0\leq k\leq \lfloor yN\rfloor$.
All Landau symbols will be taken w.r.t.~$N\to\infty$. We then in particular have
$$ \frac{1}{1-\frac{k}{N}}=O(1) \quad\text{and}\quad N-k\to\infty.$$

If we use the relation (\ref{asympt-zeros}) between the increasing
positive zeroes $j_{\alpha,r}$ of  $J_\alpha$ and $\theta_{r,N}$ of $P_N^{(\alpha,\beta)}(\cos\theta)$, we obtain for the associated zeros
$z_{N-r+1,N}$ of  $P_N^{(\alpha,\beta)}(z)$ in reversed order that
\begin{align}\label{eq:zNJacobiasymptotics}
z_{N-r+1,N}&=1-\frac{j_{\alpha,r}^2}{2N^2}+O(N^{-3}) ,
 \\ 1-z_{N-r+1,N}^2&=\frac{j_{\alpha,r}^2}{N^2}+O(N^{-3}) \quad\text{and}\quad
\sqrt{1-z_{N-r+1,N}^2}=\frac{j_{\alpha,r}}{N}+O(N^{-2}).\notag\end{align}

We next analyze the asymptotic behavior of the coefficients $a_{N-k},a_{N-k-1},b_{N-k-1}$ in the three-term recurrence
(\ref{eq:3termjacobidualnorm2}). By (\ref{recursion-jabobi-normalized}) we may write $a_N$ as $a_N=\sqrt{ \tilde a_N}$ with
$$\tilde a_N:= \frac{4N(N+\alpha+\beta)(N+\alpha)(N+\beta)}{(2N+\alpha+\beta+1)(2N+\alpha+\beta-1)(2N+\alpha+\beta)^2}.$$
Hence, after some standard manipulations,
\begin{align}
 \tilde a_N =& \frac{(1+(\alpha+\beta)/N)(1+\alpha/N)(1+\beta/N)}{4(1+(\alpha+\beta+1)/(2N))(1+(\alpha+\beta-1)/(2N))(1+(\alpha+\beta)/(2N))^2}\notag\\
    =&	\frac{1}{4}\Bigl(1-\frac{2(\alpha^2+\beta^2)-1}{4N^2}+O(N^{-3})\Bigr).
\end{align}
Therefore,
\begin{align}\label{eq:taylorJHEaNk}
	2a_{N-k}&=1-\frac{2(\alpha^2+\beta^2)-1}{8N^2\lkn^2}+O(N^{-3}),\notag\\
	2a_{N-k-1}&=1-\frac{2(\alpha^2+\beta^2)-1}{8N^2\left(1-\frac{k+1}{N}\right)^2}+O(N^{-3}).
\end{align}
In a similar way,   
\begin{equation}\label{eq:taylorJHEbNk1}
	2b_{N-k-1}=-\frac{\alpha^2-\beta^2}{2N^2}\frac{1}{\lkn^2}+O(N^{-3}).
\end{equation}
If we use \eqref{eq:zNJacobiasymptotics},\eqref{eq:taylorJHEaNk}, and \eqref{eq:taylorJHEbNk1}, and if we use the abbreviation
$$q_k:= Q_{k,N}(z_{N-r+1,N}),$$
we  rewrite
 the three-term recurrence relation \eqref{eq:3termjacobidualnorm2} for $x=z_{N-r+1,N}$ as
\begin{align*}
	2\left(1-\frac{j_{\alpha,r}^2}{2N^2}+O(N^{-3})\right)q_k&=q_{k-1}\left(1-\frac{2(\alpha^2+\beta^2)-1}{8N^2\left(1-\frac{k}{N}\right)^2}+O(N^{-3})\right)\\
	&-q_k\left(\frac{\alpha^2-\beta^2}{2N^2}\frac{1}{\lkn^2}+O(N^{-3})\right)\\
	&+q_{k+1}\left(1-\frac{2(\alpha^2+\beta^2)-1}{8N^2\left(1-\frac{k+1}{N}\right)^2}+O(N^{-3})\right).
\end{align*}
If we rearrange the terms in this equation, we obtain
\begin{align}\label{eq:3termJHEtaylored}
	q_{k+1}-2q_k+q_{k-1}&=\frac{q_{k-1}}{8N^2}\frac{2(\alpha^2+\beta^2)-1}{\lkn^2}+\frac{q_k}{2N^2}\frac{\alpha^2-\beta^2}{\lkn^2}\notag\\
	&+\frac{q_{k+1}}{8N^2}\frac{2(\alpha^2+\beta^2)-1}{\left(1-\frac{k+1}{N}\right)^2}-\frac{q_k}{N^2}j_{\alpha,r}^2\\
&+O(N^{-3})q_{k-1}+O(N^{-3})q_{k}+O(N^{-3})q_{k}.
\notag	
\end{align}
We next want to take a double sum of \eqref{eq:3termJHEtaylored}  in order to get an integral equation in the limit.
For this, some estimations are necessary.
The Cauchy-Schwarz inequality, \eqref{eq:zNJacobiasymptotics} and the orthogonality of \eqref{eq:TNorthogonalJHE} lead to 
\begin{align}\label{eq:JHEerrorsumqk}
  \sum_{k=0}^{N-1}|q_k|\leq\sqrt{\sum_{k=0}^{N-1}q_k^2}\sqrt{\sum_{k=0}^{N-1}1}={\frac{1}{\sqrt{1-z_{N-r+1,N}^2}}}\sqrt{N}=O(N^{\frac 32})
\end{align}
and $1=\sum_{i=1}^{N}(1-z_r^2)q_k^2$. In particular, we have 
\begin{equation}\label{eq:JHEerrqk} q_k=O({N})\end{equation}
uniformly in $k$.
The estimation \eqref{eq:JHEerrorsumqk} and $q_{-1}=0$ in particular ensure that the summation over the terms in the last line of
\eqref{eq:3termJHEtaylored}  satisfy
\begin{align}\label{eq:JHEsomeerrorterms}
\sum_{k=0}^{l}\big(O(N^{-3}) q_{k-1}+O(N^{-3}) q_k+O(N^{-3}) q_{k+1}\big)=O(N^{-\frac32}),\notag\\
\sum_{l=0}^{m-1}\sum_{k=0}^{l}\big(O(N^{-3})q_{k-1}+O(N^{-3})q_k+O(N^{-3})q_{k+1}\big)=O(N^{-\frac12})
\end{align}
uniformly for all $l,m\le \lfloor yN\rfloor$.
On the other hand, as $q_{-1}=0$ and $q_0=\sqrt 2/\sqrt N$,
for $0\leq m\leq \lfloor yN\rfloor$, the 
 left-hand side in \eqref{eq:3termJHEtaylored} satisfies
\begin{align}\label{eq:doublesumLHSJHE}
	\begin{split}\sum_{l=0}^{m-1}\sum_{k=0}^{l}\left(q_{k+1}-2q_k+q_{k-1}\right)&=\sum_{l=0}^{m-1}\left((q_{l+1}-q_l)-(q_0-q_{-1})\right)\\
	&=q_m-{(m+1)}q_0\\&=q_m-\frac{\sqrt 2(m+1)}{\sqrt{N}}+O(1/N).\end{split}
\end{align}
for $0\leq m\leq \lfloor yN\rfloor$,
We next consider the the first two lines of the right hand side of \eqref{eq:3termJHEtaylored}.
For this we use
\begin{align*}
	\frac{1}{\left(1-\frac{k+1}{N}\right)^2}=\frac{1}{\lkn^2}+O(N^{-1}).
\end{align*}
This and (\ref{eq:JHEerrqk}) then imply that
for
  $0\leq l\leq\lfloor yN\rfloor$,
\begin{align}\label{eq:indexshiftRHSJHE1}\begin{split}
	\sum_{k=0}^{l}\frac{q_{k-1}}{8N^2}&\frac{2(\alpha^2+\beta^2)-1}{\lkn^2}\\=&\sum_{k=0}^{l}\left(	\frac{q_{k-1}}{8N^2}\frac{2(\alpha^2+\beta^2)-1}{\left(1-\frac{k-1}{N}\right)^2}+O(N^{-2})\right)\\
	=&\sum_{k=0}^{l}\left(	\frac{q_{k}}{8N^2}\frac{2(\alpha^2+\beta^2)-1}{\left(1-\frac{k}{N}\right)^2}\right)-\frac{q_{l}}{8N^2}\frac{2(\alpha^2+\beta^2)-1}{\left(1-\frac{l}{N}\right)^2}+O(N^{-1})\\
	=&\sum_{k=0}^{l}\left(	\frac{q_{k}}{8N^2}\frac{2(\alpha^2+\beta^2)-1}{\left(1-\frac{k}{N}\right)^2}\right)+O(N^{-1}).
\end{split}\end{align}
Furthermore,
\begin{align}\label{eq:indexshiftRHSJHE2}
\begin{split}
	\sum_{k=0}^{l}\frac{q_{k+1}}{8N^2}&\frac{2(\alpha^2+\beta^2)-1}{\left(1-\frac{k+1}{N}\right)^2}\\
	=&\sum_{k=0}^{l}\frac{q_{k}}{8N^2}\frac{2(\alpha^2+\beta^2)-1}{\left(1-\frac{k}{N}\right)^2}+\frac{q_{l+1}}{8N^2}\frac{2(\alpha^2+\beta^2)-1}{\left(1-\frac{k+1}{N}\right)^2}-\frac{q_{0}}{8N^2}\frac{2(\alpha^2+\beta^2)-1}{\left(1-\frac{0}{N}\right)^2}\\
	=&\sum_{k=0}^{l}\frac{q_{k}}{8N^2}\frac{2(\alpha^2+\beta^2)-1}{\left(1-\frac{k}{N}\right)^2}+O(N^{-1}).
\end{split}\end{align}
We now use \eqref{eq:JHEsomeerrorterms}, \eqref{eq:indexshiftRHSJHE1} and \eqref{eq:indexshiftRHSJHE2} and abtain 
\begin{align}\label{eq:doublesumRHSJHE}
\begin{split}
  &\sum_{l=0}^{m-1}\sum_{k=0}^{l}\frac{q_{k}}{8N^2}\frac{2(\alpha^2+\beta^2)-1}{\lkn^2}+
  \sum_{l=0}^{m-1}\sum_{k=0}^{l}\frac{q_k}{2N^2}\frac{\alpha^2-\beta^2}{\lkn^2}\\
+&\sum_{l=0}^{m-1}\sum_{k=0}^{l}\frac{q_{k}}{8N^2}\frac{2(\alpha^2+\beta^2)-1}{\left(1-\frac{k}{N}\right)^2}-\sum_{l=0}^{m-1}\sum_{k=0}^{l}\frac{q_k}{N^2}j_{\alpha,r}^2+O(1)\\
&\quad =\sum_{l=0}^{m-1}\sum_{k=0}^{l}\frac{q_{k}}{N^2}\left(\frac{\alpha^2-\frac14}{\lkn^2}-j_{\alpha,r}^2\right)+O(1).
\end{split}\end{align}
In order to obtain a relation for the step functions $f_N$ in \eqref{def:fNJHE}, we multiply both sides in 
\eqref{eq:doublesumLHSJHE} and \eqref{eq:doublesumRHSJHE}
 with
\begin{align}\label{eq:asymptotic_zeros}\sqrt{N}\sqrt{1-z_{N-r+1,N}^2}=\frac{j_{\alpha,r}}{\sqrt{N}}+O(N^{-\frac32}).\end{align}
\eqref{eq:doublesumLHSJHE}, \eqref{eq:doublesumRHSJHE}, and \eqref{eq:3termJHEtaylored} then lead to
\begin{equation}\label{eq:3termtayloredJHE}
	\sqrt{N}\sqrt{1-z_r^2}q_m-j_{\alpha,r}\frac{\sqrt 2(m+1)}{N}+O(N^{-\frac12})
  =\sum_{l=0}^{m-1}\sum_{k=0}^{l}\frac{\sqrt{N}\sqrt{1-z_r^2}q_{k}}{N^2}\left(\frac{\alpha^2-\frac14}{\lkn^2}-j_{\alpha,r}^2\right).
\end{equation}
We no  rewrite \eqref{eq:3termtayloredJHE} as an integral equation for $f_N(y)$.
For fixed $y\in (0,1)$ consider the integers $m=m(y,N)$ with
$$\frac{m(y,N)}{N}\leq y<\frac{m(y,N)+1}{N}.$$
We now use the Landau symbol $O(N^{-j})$  for $N\to\infty$ such that the results hold  locally uniformly
 with respect to $y$. With the help of \eqref{eq:JHEerrqk} we have
\begin{align}
\begin{split}\label{eq:errorJHE}
\frac{m+1}{N}&=y+\left(\frac{m+1}{N}-y\right)=y+O(N^{-1}),\\
1-\frac{m}{N}&=1-y+y-\frac{m}{N}=1-y+O(N^{-1}),\\
\frac{1}{1-\frac{m}{N}}&=\frac{1}{1-y+O(N^{-1})}=\frac{1}{1-y}+O(N^{-1}).
\end{split}
\end{align}
We now claim that the double sum in \eqref{eq:3termtayloredJHE} can be written as an double integral as follows:
\begin{align}\begin{split}\label{eq:errordoublesumJHE}
		&\frac{1}{N^2}\sum_{l=0}^{m-1}\sum_{k=0}^{l}{\sqrt{N}\sqrt{1-z_r^2}q_{k}}\left(\frac{\alpha^2-\frac14}{\lkn^2}-j_{\alpha,r}^2\right)\\
		&\quad =\int_{0}^{y}\int_{0}^{s}{f_N(t)}\left(\frac{\alpha^2-\frac14}{(1-t)^2}-j_{\alpha,r}^2\right)dtds
		+O(N^{-1}).
	\end{split}
\end{align}
To show this consider the function
\begin{equation}g_N(t):=\sum_{k=0}^{N-1}t_{k,N}\mathbf 1_{\left[\frac{k}{N},\frac{k+1}{N}\right]}(t) \quad\text{with}\quad
 t_{k,N}:={\sqrt{N}\sqrt{1-z_r^2}q_{k}}\left(\frac{\alpha^2-\frac14}{\lkn^2}-j_{\alpha,r}^2\right).
\end{equation}
The double integral of $g_N$ then is given by
\begin{align*}
	\int_{0}^{y}\int_{0}^{s}g_N(t)dtds&=\int_{0}^{y}\int_{0}^{s}\sum_{k=0}^{N-1}t_{k,N}1_{\left[\frac{k}{N},\frac{k+1}{N}\right]}(t)dt ds\\
	&=\sum_{k=0}^{\lfloor yN\rfloor -1}\left(t_{k,N}\int_{\frac{k}{N}}^{\frac{k+1}{N}}\int_{\frac{k}{N}}^{s}1dtds+t_{k,N}\int_{\frac{k+1}{N}}^{y}\int_{\frac{k}{N}}^{\frac{k+1}{N}}1dtds\right)	\\
	&\quad +t_{\lfloor yN \rfloor,N}\int_{\frac{\lfloor yN \rfloor}{N}}^{y}\int_{\frac{\lfloor yN \rfloor}{N}}^{s}1dt ds\\
	&=\sum_{k=0}^{\lfloor yN\rfloor -1}\left(t_{k,N}\left(\frac{1}{2N^2}+\frac{1}{N}\left(y-\frac{k+1}{N}\right)\right)\right)\\
	&\quad +\frac12\left(\frac{yN-\lfloor yN\rfloor}{N}\right)^2t_{\lfloor yN\rfloor,N}\\
	&=\frac{1}{N^2}\sum_{k=0}^{\lfloor yN\rfloor -1}(\lfloor yN\rfloor-k)t_{k,N}+\frac{yN-\lfloor yN\rfloor -\frac12}{N}\frac1N\sum_{k=0}^{\lfloor yN\rfloor -1}t_{k,N}\\
	&\quad  +\frac12\left(\frac{yN-\lfloor yN\rfloor}{N}\right)^2t_{\lfloor yN\rfloor,N}.
\end{align*}
This, together with the identity
\begin{align*}
	\sum_{l=0}^{\lfloor yN\rfloor -1}\sum_{k=0}^{l}t_{k,N}=\sum_{k=0}^{\lfloor yN\rfloor-1}(\lfloor yN\rfloor-k)t_{k,N}
\end{align*}
yields that
\begin{align*}
		&\int_{0}^{y}\int_{0}^{s}g_N(t)dtds-\frac{1}{N^2}\sum_{l=0}^{\lfloor yN\rfloor -1}\sum_{k=0}^{l}t_{k,N}\\
		&=\frac12\left(\frac{yN-\lfloor yN\rfloor}{N}\right)^2t_{\lfloor yN\rfloor,N}+\frac{yN-\lfloor yN\rfloor -\frac12}{N}\frac1N\sum_{k=0}^{\lfloor yN\rfloor -1}t_{k,N}\\
		&=O(N^{-1}).
\end{align*}
Note that for the last equality, the estimation \eqref{eq:JHEerrorsumqk},\eqref{eq:JHEerrqk} and  \eqref{eq:asymptotic_zeros} were used. This shows \eqref{eq:errordoublesumJHE}.

With the error bound \eqref{eq:errorJHE},  we obtain from \eqref{eq:3termtayloredJHE} and \eqref{eq:errordoublesumJHE} that
\begin{align}\label{eq:integraleqfNJHE}
	f_N(y)-\sqrt 2\cdot j_{\alpha,r}y=\int_{0}^{y}\int_{0}^{s}{f_N(t)}\left(\frac{\alpha^2-\frac14}{(1-t)^2}-j_{\alpha,r}^2\right)dtds
	+O(N^{-\frac12}).
\end{align}
We now compare $f_N$ with the solution $f$  of the integral equation
\begin{align}\label{eq:integraleqJHE}
f(y)-\sqrt 2\cdot j_{\alpha,r}y=\int_{0}^{y}\int_{0}^{s}{f(t)}\left(\frac{\alpha^2-\frac14}{(1-t)^2}-j_{\alpha,r}^2\right)dtds.
\end{align}
For this we observe that
$$
  |f(y)-f_N(y)|\le\int_{0}^{y}\left|\left(\frac{\alpha^2-\frac14}{(1-t)^2}-j_{\alpha,r}^2\right)(f(t)-f_N(t))\right|dt
  +| O(N^{-\frac12})|$$
  where
$\left|\frac{\alpha^2-\frac14}{(1-t)^2}-j_{\alpha,r}^2\right|$ is bounded on compact intervals in $[0,1[$.
We thus conclude from the lemma of Gronwall that 
\begin{align}\label{eq:JHElocalconv}
|f_N(y)-f(y)|=O(N^{-\frac12})
\end{align}
locally uniformly on $[0,1[$.
    Moreover, as $f_N$ is bounded locally uniformly on $[0,1[$, we conclude from the definition of the $f_N$
        and from (\ref{eq:zNJacobiasymptotics}) that
        \eqref{eq:JHEerrqk} can be improved to
$q_k=O(\sqrt{N})$.
If the proof above is repeated with this estimation, \eqref{eq:JHElocalconv} can be improved to
\begin{align}\label{eq:JHElocalconv2}
|f_N(y)-f(y)|=O(N^{-1}).
\end{align}
We next consider 
the integral equation \eqref{eq:integraleqJHE} which obviously is equivalent to the differential equation
\begin{align}\label{eq:ODEJHEwithinproof}
f''(y)(1-y)^2+f(y)\left(j_{\alpha,r}^2(1-y)^2-\alpha^2+\frac14\right)=0
\end{align}
with the initial values $f(0)=0$ and $f'(0)=\sqrt 2\cdot j_{\alpha,r}$.
The differential equation
\eqref{eq:BesselJHE} and a short computation show that the function $f^{(\alpha,r)}$
in the theorem is the unique solution of this initial value problem. This completes the proof.
\end{proof}

For $r=1,\ldots,N$, the step functions $f_{N,r}$ in \eqref{def:fNJHE} can be regarded as the rows of the orthogonal matrix $T_N$ in
\eqref{eq:TNorthogonalJHE}, i.e., $(f_{N,r})_{r=1,\ldots,N}$ forms an orthonormal system in $L^2([0,1])$. Theorem \ref{theorem:localconvJHE}
now  leads to the suggestion that this also holds for the sequence $(f^{(\alpha,r)})_{r\ge 1}$ of the limit functions.
We are not able to verify this directly from  Theorem 	\ref{theorem:localconvJHE}. However, this follows immediately from a known
corresponding relation for the Bessel function $J_\alpha$; see Eq. 10.22.37 in \cite{NIST}. As this orthonormality is essential
below, we state this result separately:

\begin{lemma}\label{orthonormality-mod-bessel} The functions  $f^{(\alpha, r)}$ in Eq.~\eqref{limit-besseltype} satisfy
\begin{equation}\label{limit-orthonormal}
         \int_0^1 f^{(\alpha, r)}(y)\cdot f^{(\alpha, s)}(y)\> dy=\delta_{r,s}      \quad\quad(r,s\in\mathbb N).
                  \end{equation}
\end{lemma}

\begin{proof}[Proof of Theorem \ref{asymptotics-jacobi}]
  We keep the notations from the beginning of this section, i.e., 
  $$	T_N= \left(\sqrt{1-z_{i,N}^2}Q_{j-1,N}(z_{i,N})\right)_{i,j=1,\ldots,N}   \quad\text{and}\quad
  \tilde\Sigma_N=\left(\tilde\sigma_{i,j}\right)_{i,j=1,\ldots,N}$$
  satisfy
$$T_N^T\tilde\Sigma_N T_N={diag}\left(\frac{1}{\lambda_1},\frac{1}{\lambda_2},...,\frac{1}{\lambda_N}\right) \quad\text{with}\quad
  \lambda_k=2k(2N+\ab+1-k).$$
  The covariances $\tilde\sigma_{N-r+1,N-s+1}$ then can be written as
		\begin{align*}
		  N^2 \tilde\sigma_{N-r+1,N-s+1}&=
                  \sum_{k=1}^N\frac{N}{\lambda_k}\left(\sqrt{N}\sqrt{1-\big(z_{N-r+1,N}\big)^2} Q_{k-1}^N(z_{N-r+1,N})\right)\\
                  &\quad\quad\quad\quad\cdot \left(\sqrt{N}\sqrt{1-\big(z_{N-s+1,N}\big)^2}
                  Q_{k-1}^N(z_{N-s+1,N})\right)\\
		&=\frac{1}{N}\sum_{k=0}^{N-1}\frac{N^2}{\lambda_{k+1}}\cdot f_{N,r}\left(\frac{k}{N}\right) \cdot f_{N,s}\left(\frac{k}{N}\right) \\
		  &=\frac{1}{N}\sum_{k=0}^{N-1}\frac{N^2}{2(k+1)(2N+\ab-k)}\cdot f_{N,r}\left(\frac{k}{N}\right)
                  f_{N,s}\left(\frac{k}{N}\right).
		\end{align*}
		These expressions
                can be identified as integrals of step functions in which the functions $f_{N,r}$,   $f_{N,s}$ and the step functions
                \begin{align*}
			h_N(y):&=\frac{N^2}{2(k+1)(2N+\ab-k)}{ \bf 1}_{[\frac{k}{N},\frac{k+1}{N})}(y)\\
			&=\frac{1}{2\frac{k+1}{N}\left(2+\frac{\ab}{N}-\frac kN\right)}{ \bf 1}_{[\frac{k}{N},\frac{k+1}{N})}(y) \quad(0\le  y<1 )
		\end{align*}
                with $h_N(1):=1/2$  appear. The $h_N$ clearly tend to the function
		\begin{align*}
		h(y):=	\frac{1}{2y(2-y)} 
		\end{align*}
		for $N\to\infty$ and $y\in ]0,1]$.
                      Moreover, for  $y\in ]0,1]$ and $0\leq k<N$ with  $\frac{k}{N}\leq y<\frac{k+1}{N}$, we have
		\begin{equation}\label{est-hn}
		 0\le	 h_N(y)=\frac{1}{2\frac{k+1}{N}\left(2+\frac{\ab}{N}-\frac kN\right)}\leq \min\left(\frac{N}{2},\frac{1}{2y}\right).
		\end{equation}
                We now claim that
\begin{equation}\label{end-limit}
  \lim_{N\to \infty}N^2\sigma_{N-r+1,N-s+1}=	\lim_{N\to \infty}\int_0^1 h_N(y)f_{N,r}(y) f_{N,s}(y)\> dy=
  \int_{0}^{1} h(y)f^{(\alpha,r)}(y)f^{(\alpha,s)}(y)\> dy,
\end{equation}
which together with the definition  of the $f^{(\alpha,r)}$ and the substitution $y\mapsto 1-y$ in the integral
yields the first statement of the theorem.
Please notice that the integral on the RHS of \eqref{end-limit} exists as $f^{(\alpha,r)}(0)=0$ and $f^{(\alpha,r)\prime}(0)>0$.

We now prove the second $=$ in \eqref{end-limit}. We shall do this first for $r=s$ and
consider $[0,1]$-valued continuous functions $g_\delta$ (with $\delta\in]0,1/2[$)
    with $g_\delta(y)=1$ for $y\in[2\delta,1]$ and  $g_\delta(y)=0$ for $y\in[0,\delta]$.
Then
\begin{align}\label{end-limit-est}
  \Bigl|&\int_0^1 h_N(y)f_{N,r}(y)^2 \> dy-\int_{0}^{1} h(y)f^{(\alpha,r)}(y)^2 \> dy\Bigr| 
  \notag\\
  &\le \left|\int_{0}^{1} g_\delta(y)h(y)(f_{N,r}(y)^2-f^{(\alpha,r)}(y)^2)\> dy \right|+
  \left|\int_{0}^{1} g_\delta(y)f_{N,r}(y)^2(h(y)-h_N(y))\> dy \right|\notag\\
  &\quad+ \left|\int_{0}^{2\delta}(1-g_\delta(y))h_N(y)f_{N,r}(y)^2\> dy \right|
  +\left|\int_{0}^{2\delta}(1-g_\delta(y))h(y)f^{(\alpha,r)}(y)^2\> dy \right|\notag\\
&=: I_1+I_2+I_3+I_4.
\end{align}
Now let $\epsilon>0$. Then, by the properties of $f^{(\alpha,r)}$ at $0$, $I_4\le\epsilon$ for $\delta>0$ small enough.
The same holds for  $I_3$, as \eqref{est-hn}, $f^{(\alpha,r)}(0)=0$, and \eqref{eq:JHElocalconv2} lead to
\begin{align}h_N(y)f_{N,r}(y)^2\le& h_N(y)f^{(\alpha,r)}(y)^2 +  h_N(y)|f_{N,r}(y)^2-f^{(\alpha,r)}(y)^2|\notag\\
\le& C_1\frac{1}{y}y^2 + h_N(y)|f_{N,r}(y)-f^{(\alpha,r)}(y)|(f_{N,r}(y)+f^{(\alpha,r)}(y))\notag\\
\le&  C_1y +h_N(y)|f_{N,r}(y)-f^{(\alpha,r)}(y)|(|f_{N,r}(y)-f^{(\alpha,r)}(y)|+2|f^{(\alpha,r)}(y)|)\notag\\
\le&  C_1y +C_2N\frac{1}{N} (\frac{1}{N}+y)=  C_1y +C_2(\frac{1}{N}+y)
		\end{align}
for some constants $C_1,C_2>0$.  Moreover,
as the probability measures $f_{N,r}(y)^2dy$ tend to the  probability measure $f^{(\alpha,r)}(y)^2\> dy$
vaguely on $[0,1[$ by Theorem \ref{theorem:localconvJHE}
and thus weakly, we see that $I_1\le\epsilon$ for $N$ large enough. Finally, as $h_N\to h$ uniformly on $[\delta,1]$, we also have
$I_2\le\epsilon$  for $N$ large. This proves the second $=$ in \eqref{end-limit} for $r=s$.

We next observe that this  proof of \eqref{end-limit} for $r=s$
together with the orthogonality part in Lemma \ref{orthonormality-mod-bessel} 
also implies that for $r\ne s$,
$$	\lim_{N\to \infty}\int_0^1 h_N(y)(f_{N,r}(y)+ f_{N,s}(y))^2\> dy=
\int_{0}^{1} h(y)(f^{(\alpha,r)}(y)+f^{(\alpha,s)}(y))^2\> dy.$$
This together with  \eqref{end-limit} for $r=s$ now lead to  \eqref{end-limit} for $r \ne s$.
This completes the proof of \eqref{limit-jacobi-trigonometric}.

Finally,  \eqref{limit-jacobi-algebraic} follows from \eqref{limit-jacobi-trigonometric},
 \eqref{connection-covariances},
and \eqref{eq:zNJacobiasymptotics}.
\end{proof}

\section{Hard edge covariances for frozen Laguerre ensembles}

As announced in the introduction, the approach to the limits of the extremal covariances for frozen Jacobi ensembles above
can be applied also to the hard edge analysis for frozen Laguerre ensembles, i.e., to the smallest eigenvalues.
For this we first recapitulate some facts from \cite{A, AV, AHV, V}.
In order to state a CLT similar to Theorem \ref{theoremclttrig}, we fix some parameter $\nu>0$ and consider the ordered
zeros $0< z_{1,N}^{(\nu-1)}<\ldots<  z_{N,N}^{(\nu-1)}$
of the Laguerre polynomial $L_N^{(\nu-1)}$ as defined e.g. in \cite{S}. We also use the vector
$${\bf r}:={\bf r}(\nu):=(r_1,\ldots, r_N):=\left(\sqrt{ z_{1,N}^{(\nu-1)}}, \ldots, \sqrt{ z_{N,N}^{(\nu-1)}}\right).$$
By \cite{ AV,  V, DE}, we then have the following CLT for the $\beta$-Laguerre ensembles 
\begin{equation}\label{densitiybetaLaguerre2}
d\mu_L^{\beta,\nu}(x)=	C_{\beta,\nu, N,L}\prod_{1\leq i<j\leq N}(x_j^2-x_i^2)^{2\beta}\prod_{i=1}^N\left( x_i^{2\nu\beta} e^{-x_i^2/2}\right)\> dx
\end{equation}
on  the Weyl chamber  $C_N^B:=\{x\in\mathbb R^N:\> 0\le x_1\le x_2\le \cdots\le x_N\}$:

\begin{theorem}\label{clt-laguerre-complete}
Let $\nu>0$.  For $\beta>0$ consider Laguerre-type random variables $X_\beta$ on $C_N^B$ 
  with the densities (\ref{densitiybetaLaguerre2}).
  Then  $ X_\beta-\sqrt{2\beta} \cdot {\bf r}$
 	converges  for $\beta\to\infty$ in distribution to 
        $N(0,\Sigma_N)$ where the inverse
        $\Sigma_N^{-1}=:S_N=(s_{i,j})_{i,j=1,...,N}$ of the covariance matrix $\Sigma_N$  is given by
 	\begin{equation}
 	   s_{i,j}=\begin{cases} 1+\frac{\nu}{r_i^2}+\sum_{l\ne i}
(r_i-r_l)^{-2}+\sum_{l\ne i} (r_i+r_l)^{-2} & \text{for}\quad i=j, \\
 (r_i+r_j)^{-2}-(r_i-r_j)^{-2} & \text{for}\quad i\ne j.
 	\end{cases}
 	\end{equation}
The matrix  $ S_N$ has the eigenvalues $\lambda_k=2k$ with $k=1,\ldots,N,$
 	and the $\lambda_k$ have  eigenvectors of   the form
 	$$v_k:=\left(q_{k-1}(r_1^2)\sqrt r_1,\ldots,q_{k-1}(r_N^2)\sqrt r_N\right)^T$$
 	with the finite orthonormal
        polynomials $(q_{k})_{k=0,\ldots,N-1}$  associated  with the discrete probability measure
\begin{equation}\label{orthogonality-measure-b1}
\mu_{N,\nu}:=\frac{1}{N(N+\nu-1)}(z_{1,N}^{(\nu-1)}\delta_{z_{1,N}^{(\nu-1)}}+\ldots+z_{N,N}^{(\nu-1)}\delta_{z_{N,N}^{(\nu-1)}}).
\end{equation}
\end{theorem}

We now consider the slightly renormalized orthogonal polynomials
$$\Bigl(Q_{k,N}:= \frac{1}{\sqrt{N(N+\nu-1)}}q_{k}\Bigr)_{k=0,..,N-1}$$
for which then the matrices
	\begin{equation}\label{eq:TNorthogonalLHE}	
	T_N:= \left(\sqrt z_{i,N}^{(\nu-1)}\cdot Q_{j-1,N}(z_{i,N}^{(\nu-1)})\right)_{i,j=1,\ldots,N}
        \end{equation}
        are orthogonal with $T_N^T\Sigma_N T_N=diag(\frac{1}{2},\frac{1}{4},...,\frac{1}{2N})$. Moreover, as discussed in \cite{AHV},
        these $Q_{k,N}$ are the dual orthogonal polynomial of the Laguerre polynomials $(L_k^{(\nu-1)})_{k=0,...,N}$ in the sense of 
        De Boor and Saff \cite{BS, I, VZ} which means that, by the three-term-recurrence of the Laguerre polynomials in \cite{S},
        they satisfy the three-term-recurrence
\begin{align}\label{3termlaguerredualnorm2}
	xQ_{k,N}(x)=&\sqrt{(N-k)(N-k+\nu-1)}Q_{k-1,N}(x)+(2(N-k)+\nu-2)Q_{k,N}\notag\\ 
	+&\sqrt{(N-k-1)(N-k-1+\nu-1)}Q_{k+1,N}(x)\quad(k<N)
	\end{align}
with the initial condition $Q_{-1,N}=0$ and $Q_{0,N}=\frac{1}{\sqrt{N(N+\nu-1)}}$; see also \cite{A}.
The following result is shown in principle by Andraus \cite{A} (up to the precise normalizations)
and analogous to Theorem \ref{theorem:localconvJHE}. We  point out that we here multiply the polynomials
$Q_{k,N}$ with positive leading coefficients by $(-1)^k$ as we are
interested now in $Q_{k,N}(y)$ for small positive $y$.

	\begin{theorem} \label{theorem:localconvLHE}
		Let $\nu>0$ and  $r\in\mathbb{N}$. For  $N\in\mathbb{N}$ with $N>r$ consider the step functions
		\begin{equation}\label{stepfunction-lagu}
		  f_N(y):=f_N^{[\nu,r)}(y):=
                    \sum_{k=0}^{N-1}\sqrt{N}\sqrt{z_{r,N}^{(\nu-1)}}(-1)^kQ_{k,N}(z_{r,N}^{(\nu-1)}){\bf 1}_{[\frac{k}{N},\frac{k+1}{N}[}(y).
		\end{equation}
                as well as the function
\begin{align}\label{eq:limitfLHE}
		f^{(\nu,r)}(y):=\frac{-J_{\nu-1}(j_{\nu-1,r}\sqrt{1-y})}{J_{\nu-1}^\prime(j_{\nu-1,r})}.
		\end{align} 
for $y\in[0,1)$.  Then, for $N\to\infty$, $|f_N-f^{(\nu,r)}|=O(N^{-1})$ locally  uniformly on $[0,1[$.
	\end{theorem}

        \begin{proof}	We first recapitulate that by Theorem 1.1 in \cite{G1}, for $r\in\mathbb N$,
		\begin{equation}\label{eq:z1asymptoticshardLaguerre}
		z_r:=z_{r,N}^{(\nu-1)}=\frac{j_{\nu-1,r}^2}{4N+2\nu}+O(N^{-3})=\frac{j_{\nu-1,r}^2}{4N}+O(N^{-2}).
\end{equation}
This, the  computations in  \cite{A}, and the arguments in the proof of Theorem \ref{theorem:localconvJHE} now imply that
	\begin{align}\label{eq:inteqfnLHE}
	  f_N(y)(1-y)-\frac{j_{\nu-1,r}y}{2}
          =\int_{0}^{y}\int_{0}^{s}\frac{f_N(t)}{4}\left(\frac{\alpha^2}{1-{ t}{}}-j_{\nu-1,r}^2\right)dtds-\int_{0}^{y}f_N(t)dt+O(N^{-\frac12})
	\end{align}
        for $y\in[0,1[$.
	Partial integration leads to
	\begin{equation}
	f_N(y)(1-y)-\frac{j_{\nu-1,r}y}{2}=\int_{0}^{y}\left((y-t)\frac{f_N(t)}{4}\left(\frac{(\nu-1)^2}{1-{ t}{}}-j_{\nu-1,r}^2\right)-f_N(t)\right)dt+O(N^{-\frac12}).
	\end{equation}
        This and the lemma of Gronwall now show that
	\begin{align}\label{eq:LHElocalconv}
		|f_N(y)-f(y)|=O(N^{-\frac12})
	\end{align}
for the	solution $f$  of the integral equation
	\begin{align}\label{eq:integraleqLHE}
		f(y)(1-y)-\frac{j_{\nu-1,r}y}{2}=\int_{0}^{y}\left((y-t)\frac{f(t)}{4}\left(\frac{(\nu-1)^2}{1-{ t}{}}-j_{\nu-1,r}^2\right)-f(t)\right)dt.
	\end{align}
	This implies that
        \begin{align}
	Q_{k,N}(z_{r,N}^{(\nu-1)})=O(1).
	\end{align}
	If we repeat the skipped computations above  with this estimation similar to the corresponding part
 in the proof of Theorem \ref{theorem:localconvJHE}, we conclude that \eqref{eq:LHElocalconv} can be improved to
	\begin{align*}
		|f_N(y)-f(y)|=O(N^{-1}).
	\end{align*}
	The integral equation \eqref{eq:integraleqLHE} corresponds to the initial value problem
	\begin{align*}
		f''(y)(1-y)^2-(1-y)f'(y)+f(y)\frac{1}{4}(j_{\nu-1,r}^2(1-y)-(\nu-1)^2)=0 
	\end{align*}
        for $y\in[0,1[$ with  $f(0)=0$ and $f'(0)=j_{\nu-1,r}/2$.
          It can be now easily checked with the differential equation (\ref{eq:BesselJHE}) of the Bessel functions
           that the function
           $f^{(\nu,r)}$ defined in \ref{eq:limitfLHE}
           is the unique solution of this initial value theorem; see also Lemma 3.1 of \cite{A}. This completes the proof.
        \end{proof}

        We recapitulate that the step functions $(f_N^{(\nu,r)})_{r=1,\ldots,N}$ in (\ref{stepfunction-lagu}) are orthonormal in $L^2([0,1])$.
        Moreover, similar to the preceding section,  Eq.~10.22.37 in \cite{NIST} also ensures that the limit functions $(f^{(\nu,r)})_{r\ge 1}$
        are also orthonormal in $L^2([0,1])$. We thus can copy the proof of Theorem \ref{asymptotics-jacobi},
        and obtain the following limit result from Theorem \ref{theorem:localconvLHE}. This result generalizes Theorem 1.3
 in \cite{A} slightly where there slightly different notations are used:

	\begin{theorem}\label{limit-cov-laguerre}
			Consider the covariance matrices $\Sigma_{N}=\left(\sigma_{i,j}^{(N)}\right)_{i,j=1,...,N}$ of 
		        the frozen Laguerre ensembles in Theorem \ref{clt-laguerre-complete}. Then for $r,s\in\mathbb{N}$
		\begin{align*}
		\lim_{N\to\infty} N\sigma_{r,s}^{(N)}=
		\int_{0}^{1}\frac{ J_{\nu-1}(j_{\nu-1,r}\sqrt{1-y}) J_{\nu-1}(j_{\nu-1,s}\sqrt{1-y})}{2y\cdot J_{\nu-1}^\prime(j_{\nu-1,r})J_{\nu-1}^\prime(j_{\nu-1,s})} \>    dy.
		\end{align*}
\end{theorem}
                
        \begin{proof}  We use the notations of Theorem \ref{theorem:localconvLHE} as well as
          $$T_N^T\Sigma_N T_N=diag\left(\frac{1}{2},\frac{1}{4},...,\frac{1}{2N}\right).$$
          Hence,the covariances $\sigma_{r,s}^{(N)}$ can be written as
			\begin{align*}
			  N\sigma_{r,s}^{(N)}&=\sum_{k=1}^N\frac{1}{2k}\left(\sqrt{N}\sqrt{z_{r,N}^{(\nu-1)}}Q_{k-1}^N(z_{r,N}^{(\nu-1)})\right)
                          \left(\sqrt{N}\sqrt{z_{s,N}^{(\nu-1)}}Q_{k-1}^N(z_{s,N}^{(\nu-1)})\right)\\
				&=\frac{1}{N}\sum_{k=0}^{N-1}\frac{1}{\frac{2(k+1)}{N}}f_N^{(\nu,r)}(k/N)f_N^{(\nu,s)}(k/N).
			\end{align*}
                        These sums can be regarded as integrals over step functions on $[0,1[$, where
                            Theorem \ref{theorem:localconvLHE} implies that the integrands converge locally uniformly on $[0,1[$ to
$$\frac{ J_{\nu-1}(j_{\nu-1,r}\sqrt{1-y}) J_{\nu-1}(j_{\nu-1,s}\sqrt{1-y})}{2y\cdot J_{\nu-1}^\prime(j_{\nu-1,r})J_{\nu-1}^\prime(j_{\nu-1,s})},$$
                                  The arguments of the proof of Theorem \ref{asymptotics-jacobi} imply
                                  that we again can interchange the limit $N\to\infty$ with the integration
                                  which then leads to the claim.
\end{proof}

Theorem \ref{limit-cov-laguerre-intro} is now  a direct consequence from  Theorem  \ref{limit-cov-laguerre}.

        We finally remark that a generalization of Theorem \ref{limit-cov-laguerre} with a  different proof
        can be found in \cite{L}. We  notice that the limit in Theorem 1 of \cite{L} contains an additional factor
         $j_{\alpha,r} j_{\alpha,s}$ which is caused by  different coordinates of the Laguerre ensembles 
        similar to the different limits in Theorem \ref{asymptotics-jacobi} for different coordinates.

\end{document}